\documentclass[10pt]{amsart}

\usepackage[colorlinks=true, urlcolor=blue, linkcolor=blue, citecolor=blue]{hyperref}

\usepackage{amsmath, amsthm, amssymb}

\usepackage[margin=2cm]{geometry}

\usepackage[style=numeric,sorting=nty,url=true,giveninits=true,maxcitenames=100,maxbibnames=100]{biblatex}
\renewbibmacro{in:}{}

\usepackage{verbatim}

\usepackage{color}

\usepackage[shortlabels]{enumitem}

\allowdisplaybreaks[2]

\renewcommand{\S}{\mathbb{S}}

\newtheorem{thm}{Theorem}[section]
\newtheorem*{thm*}{Theorem}
\newtheorem{prop}[thm]{Proposition}
\newtheorem{lem}[thm]{Lemma}
\newtheorem{cor}[thm]{Corollary}

\newtheorem{conjecture}{Conjecture}

\theoremstyle{definition}
\newtheorem{rmk}[thm]{Remark}

\begin{document}

\title{On an inverse curvature flow in two-dimensional space forms}

\author[K.-K. Kwong]{Kwok-Kun Kwong}

\address{University of Wollongong\\
Northfields Ave\\
2522 NSW, Australia}
\email{kwongk@uow.edu.au}
\email{glenw@uow.edu.au}
\email{vwheeler@uow.edu.au}

\author[Y. Wei]{Yong Wei}
\address{School of Mathematical Sciences, University of Science and Technology of China, Hefei 230026, P.R. China}
\email{yongwei@ustc.edu.cn}

\author[G. Wheeler]{Glen Wheeler}
\author[V.-M. Wheeler]{Valentina-Mira Wheeler}

\thanks{}


\begin{abstract}
We study the evolution of compact convex curves in two-dimensional space forms.
The normal speed is given by the difference of the weighted inverse curvature
with the support function, and in the case where the ambient space is the
Euclidean plane, is equivalent to the standard inverse curvature flow.
We prove that solutions exist for all time and converge exponentially fast in
the smooth topology to a standard round geodesic circle.
This has a number of consequences: first, to prove the isoperimetric
inequality; second, to establish a range of weighted geometric inequalities;
and third, to give a counterexample to the $n=2$ case of a conjecture of
Gir\~ao-Pinheiro.
\end{abstract}
\maketitle

\section{Introduction}

Let $M^{2}(K)$ be a real simply connected $2$-dimensional space form with
constant sectional curvature equal to $K\in\{-1,0,1\}$.
We view $M^2(K)$ as a warped product manifold $I\times \mathbb{S}^1$ equipped
with the metric
\begin{align*}
\-g=dr^2+\phi(r)^2 d\theta^2,
\end{align*}
where $\phi$ satisfies $\phi''=-K \phi$, $\phi(0)=0$, and $\phi'(0)=1$.
On $M^2(K)$ there is a conformal Killing field $V$ satisfying $\overline\nabla
V = \phi'(r)g$ (see Section 2).
This conformal Killing field is used for the support function.

Let $X_0:\mathbb{S}^1\to M^2(K)$ be a closed convex curve containing the origin
$0 = \{r=0\}\in M^2(K)$.
If $K=1$, we also assume $X_0$ lies in the open hemisphere centered at $0$.
We consider the flow $X:\S^1\times[0,T)\rightarrow M^2(K)$ where
\begin{align}\label{flow}
\begin{cases}
\frac{\partial}{\partial t} X(x, t)=\left(\frac{\phi^{\prime}(r)}{\kappa}-u\right) \nu(x, t) \\
X(\cdot, 0)=X_{0},
\end{cases}
\end{align}
where $\kappa$ is the curvature, $\nu$ is the unit outward normal, and
$u=\langle V, \nu\rangle $ is the (generalised) support function of the curve. Note that when
$K=0$, the flow \eqref{flow} is, up to rescaling \cite{And98}, equivalent to
the standard inverse curvature flow
\begin{equation}
\frac{\partial}{\partial t}X=\frac{1}{\kappa}\nu
\,.
\label{EQicf}
\end{equation}
Local existence for the flows \eqref{flow} with smooth data can be obtained in
a standard way (for instance) by using classical quasilinear PDE theory and
writing the solution as a graph over the initial curve.

Our main result is that each of the flows \eqref{flow} exist globally in time,
and converge, as $t\rightarrow\infty$, to a geodesic circle exponentially fast
in the smooth topology.

\begin{thm}\label{s4.thm}
Let $X_0:\S^1\to M^2(K)$ be a smooth closed and strictly convex curve containing the origin.
There exists a unique solution $X:\mathbb S^1\times[0, \infty)\to M^2(K)$ of
the flow \eqref{flow} such that
\begin{enumerate}
\item[(a)] $X(\cdot, 0)=X_0$
\item[(b)] $X(\cdot, t)$ is a smooth strictly convex curve which is star-shaped for all $t$;
\item[(c)] There is a unique standard smooth geodesic circle $X_\infty(\mathbb
	S^1)$ centered at the origin, such that
\[
	\lim_{t\to\infty} X(\mathbb S^1, t) = X_\infty(\mathbb S^1)
\]
exponentially fast in the smooth topology.
\end{enumerate}
\end{thm}

The expansion of curves in the plane by the expansion flow \eqref{EQicf} is by
now a classical topic in curvature flow: see
\cite{gerhardt,urbas} and their many citations.
For curves in particular we mention the recent note \cite{K19}, that applies a
technique from Andrews-Bryan \cite{AB11} to very efficiently deduce convergence
to a circle for solutions to \eqref{EQicf}.

Inverse curvature flow in hyperbolic space and in the sphere are again the
subject of seminal work by Gerhardt \cite{gerhardtsphere,gerhardthyperbolic}.
We invite the reader to explore the citations of these works.
The particular kind of generalisation of the inverse curvature flow that we
study here \eqref{flow} enjoys uniform upper and lower bounds on curvature.
This is due to the presence of the support function term, which is related (as
explained by Andrews \cite{And98}) by continuous rescaling to the inverse
curvature flow in the plane.
In other space forms, the flow \eqref{flow} is more convenient to study
compared to \eqref{EQicf}; in this sense, we say that the flow \eqref{flow} is
adapted to the space form setting.
This is well-explained in Brendle-Guan-Li \cite{BGL}, where the flow is introduced, and results
given for hypersurfaces with dimension $n\ge2$.
This flow was studied in higher dimensions with hyperbolic background in
\cite{HLW}.
For rigidity results and geometric inequalities in the sphere, we refer to
\cite{scheuer16}.
The case of curves has been left open in earlier work.
Indeed, the analysis of the flows must be done differently in low dimensions.
This is the primary focus of our paper: to fill this gap.

There is a strong tradition of application for curvature flow to geometric inequalities.
See for instance the classical use of the inverse mean curvature flow by Huisken-Ilmanen to establish the Riemannian Penrose inequality \cite{HI01}.
The isoperimetric inequality is another typical application (also one that we perform here, for both convex and non-convex cases -- see Theorem \ref{TMiso}).
There has been recently a burst of activity in the area, including inequalities
involving quermassintegrals, Minkowski-type inequalities, Alexandrov-Fenchel inequalities and more general weighted inequalities \cite{BHW,CGLS,CS21,dLG,GWW,GL2,GL,li2014geometric,HWZ,HL19,HLW,KM,KM2,S20}.

In our work here, additional to the convergence result above we give some applications of the flows \eqref{flow}.
These include the aforementioned isoperimetric inequality (Theorem \ref{TMiso}), and weighted geometric inequalities:

\begin{thm}\label{thmintPhi}
	Let $\gamma$ be a smooth, closed and convex curve in $\mathbb H^2$, $\mathbb S^2_+$ or $\mathbb R^2$. (Here $\S^2_+$ denotes any hemisphere.) Then we have
\begin{align}\label{s5.2-1}
\int_{\gamma} \Phi(r) \kappa\,ds \ge \frac{1}{2\pi}\left(L^{2}-2 \pi A\right),
\end{align}
where $\Phi(r)$ is given by
\begin{align}\label{Phi}
\Phi(r)=\int_{0}^{r} \phi(s) d s=
\begin{cases}
\cosh r-1, & K=-1 \\
\frac{r^{2}}{2}, & K=0 \\
1-\cos r, & K=1
\end{cases}
\,.
\end{align}
Equality holds if and only if $\gamma$ is a geodesic circle centered at the origin.
\end{thm}

This weighted inequality has a few important consequences that we briefly mention.
Combining Theorem \ref{thmintPhi} with the Guass-Bonnet theorem generalises Theorem 1 from \cite{dLG} (see Corollary \ref{CY1}).
Furthermore, in the case of $K=1$, we find
\[
	\int_{\gamma} \kappa\cos r \le 2 \pi-\frac{L^{2}}{2 \pi}
	\,.
\]
This inequality can be used to show that a conjecture of Gir\~ao-Pinheiro \cite{girao2017alexandrov} is false for $n=2$ (see Corollary \ref{CY2}).

The paper is organised as follows.
In Section 2 we give evolution equations for general flows of curves in space forms, which are then analysed in Section 3 for the flows \eqref{flow} in particular.
There, we establish monotonicity of the isoperimetric deficit, two-sided bounds for the curvature scalar, and bounds on the height function when written as a radial graph.
For these bounds a novel approach is taken, that requires an auxiliary function.
In Section 4 we study the long time behaviour of the flow.
Given our estimates from Section 3, global existence follows using a standard argument.
A weak kind of subsequential convergence follows from monotonicity of the isoperimetric deficit and our estimates.
This can then be strengthened to full, exponentially fast convergence by using a linearisation technique.
The proof of Theorem \ref{s4.thm} is completed in Section 4.
Section 5 is concerned with applications of the flows \eqref{flow}, in
particular the isoperimetric inequality, the weighted geometric inequalities,
and its corollaries mentioned above.

\section*{Acknowledgements}

Kwok-Kun Kwong would like to thank Man-Chun Lee for useful discussions.
Yong Wei was supported by National Key Research and Development Project
SQ2020YFA070080 and a research grant from University of Science and Technology
of China.
Valentina-Mira Wheeler was supported in part by Discovery Project DP180100431
and DECRA DE190100379 of the Australian Research Council.

\section{General evolution equations}

Let us briefly take some time to make more concrete the setting.
If $K=0$, $\phi(r)=r$ and $I=[0, \infty)$, $M^{2}$ is the Euclidean plane
$\mathbb{R}^{2}$. If $K=-1$, $\phi(r)=\sinh r$ and $I=[0, \infty)$, $M^{2}$ is
the Euclidean plane $\mathbb{R}^{2}$. If $K=1$, we have $\phi(r)=\sin r$.
However, for technical reasons we restrict ourselves to $I=[0, \frac{\pi}{2})$
and so $M^{2}(1)=\mathbb S^2_+$ is the open hemisphere. (We will see that under
a convexity assumption the curve will stay within $\mathbb S^2_+$ along the
flow.)
Let $\overline \nabla $ be the connection on $M^2(K)$. It is well-known that
the vector field $V= \overline \nabla \Phi=\phi(r)\partial_{r}$ on $M^{2}(K)$
(here $\Phi$ is as defined in \eqref{Phi})
is a conformal Killing field with $\overline \nabla V =\phi'(r) \bar{g}$.
Given a closed convex curve $X_0:\mathbb{S}^1\to M^2(K)$, we consider the following general flow
\begin{equation}\label{s2. flow1}
\begin{split}
\begin{cases}
	\frac{\partial}{\partial t}X(x, t) &= F \nu (x, t)\\
	X(\cdot, 0) &= X_0,
\end{cases}
\end{split}
\end{equation}
in $M^2(K)$, where $F$ is a smooth function on the evolving curves $\gamma_t=X(\mathbb{S}^1, t)$. We derive the evolution equations based on the calculation as in \cite{Gr89}.

Let $X(\theta, t)$, $\theta \in \mathbb S^{1}$ be the position of the evolving curves $\gamma_{t}=X\left(\mathbb S^{1}, t\right)$ and $\vartheta=|X_\theta|$. We define the arc-length parameter $s$ by
\begin{equation*}
\frac{\partial}{\partial s}=\frac 1{\vartheta} \frac{\partial}{\partial\theta}.
\end{equation*}
Denote the unit tangent vector by
\begin{equation*}
T=\frac{\partial X}{\partial s}=\frac 1{\vartheta}X_\theta.
\end{equation*}
The arclength element of the curve $\gamma_t$ is represented by $ds=\vartheta d\theta$.

\begin{lem}
Along the flow \eqref{s2. flow1}, the length $L$ of the evolving curves satisfies
\begin{align}\label{L'}
\frac d{dt} L(t) =& \int_{\gamma_t}\left\langle \frac{\partial}{\partial t}X, \kappa \nu \right\rangle ds.
\end{align}
\end{lem}
\begin{proof}
Denote $W=\partial_tX$. Then
\begin{align*}
\frac d{dt} L(t) =& \frac d{dt}\int_{\mathbb{S}^1}\vartheta d\theta\\
=& \int_{\mathbb{S}^1}W(|X_\theta|)d\theta \\
= &\int_{\mathbb{S}^1}\frac 1{|X_\theta|}\langle \overline{\nabla}_WX_\theta, X_\theta\rangle d\theta\\
=&\int_{\mathbb{S}^1}\langle \overline{\nabla}_{X_\theta}W, T\rangle d\theta.
\end{align*}
Since $W$ is parallel to the normal direction, then
\begin{align*}
\frac d{dt} L(t) =&- \int_{\mathbb{S}^1}\langle W, \overline{\nabla}_{X_\theta}T\rangle d\theta\\
=& -\int_{\gamma_t}\langle W, \overline{\nabla}_{s}T\rangle \vartheta d\theta\\
=& \int_{\gamma_t}\langle W, \kappa \nu \rangle ds.
\end{align*}
\end{proof}

\begin{lem}
The evolution of $\vartheta=|X_\theta|$ is given by
\begin{equation*}
\frac{\partial}{\partial t}\vartheta =F\kappa \vartheta.
\end{equation*}
\end{lem}
\begin{proof}

By differentiating in time $t$, we have
\begin{align*}
2\vartheta \frac{\partial}{\partial t}\vartheta =& 2\left\langle \frac{\partial}{\partial t} \frac{\partial X}{\partial \theta}, \frac{\partial X}{\partial \theta}\right\rangle \\
= & 2\left\langle \frac{\partial}{\partial\theta} \frac{\partial X}{\partial t}, \frac{\partial X}{\partial \theta}\right\rangle\\
=& 2\left\langle \frac{\partial}{\partial\theta} (F \nu), \frac{\partial X}{\partial \theta}\right\rangle\\
=&2\left\langle F \frac{\partial}{\partial\theta} \nu, X_\theta\right\rangle \\
=&2F \vartheta^2 \left\langle \nu_s, T\right\rangle \\
=& 2F\kappa \vartheta^2.
\end{align*}
\end{proof}

\begin{lem}
Covariant derivative with respect to $s$ and $t$ are related by
\begin{equation*}
\overline{\nabla}_t\overline{\nabla}_s=\overline{\nabla}_s\overline{\nabla}_t+F \bar{R}\left(T, \nu \right)
\end{equation*}
where $\bar{R}$ denotes the ambient curvature.
\end{lem}
\begin{proof}
We have
\begin{align*}
\overline{\nabla}_t\overline{\nabla}_\theta=&\overline{\nabla}_\theta\overline{\nabla}_t+ \bar{R} (\frac{\partial}{\partial \theta}, \frac{\partial}{\partial t}).
\end{align*}
Then
\begin{align*}
\overline{\nabla}_t\overline{\nabla}_s=&\frac{\partial}{\partial t}\left(\frac{1}{\vartheta}\right)\overline{\nabla}_\theta+\frac{1}{\vartheta}\overline{\nabla}_t\overline{\nabla}_\theta\\
=&\frac{-F\kappa}{\vartheta}\overline{\nabla}_\theta+\frac{1}{\vartheta}\left(\overline{\nabla}_\theta\overline{\nabla}_t+\bar{R}\left(\frac{\partial}{\partial \theta}, \frac{\partial}{\partial t}\right)\right)\\
=&\overline{\nabla}_s\overline{\nabla}_t-F\kappa\overline{\nabla}_s+ \bar{R}(T, F\nu)
\,.
\end{align*}
\end{proof}

\begin{lem}
The derivative of $T$ with respect to time $t$ is
\begin{equation*}
\overline{\nabla}_tT=\overline{\nabla}_s F \nu.
\end{equation*}
This is equivalent to
\begin{equation}\label{nu'}
\overline{\nabla}_t \nu =-\overline{\nabla}_sF T.
\end{equation}
\end{lem}
\begin{proof}
\begin{align*}
\overline{\nabla}_{t} T &=\overline{\nabla}_{t}\left(\frac{1}{\vartheta} \frac{\partial X}{\partial \theta}\right) \\
&=-\frac{1}{\vartheta^{2}} \overline{\nabla}_{t} \vartheta \frac{\partial X}{\partial \theta}+\frac{1}{\vartheta} \frac{\partial}{\partial t} \frac{\partial X}{\partial \theta} \\
&=-F \kappa T+\overline{\nabla}_{s}(F \nu) \\
&=-F \kappa T+\overline{\nabla}_{s} F \nu +F \kappa T \\
&=\overline{\nabla}_{s} F \nu.
\end{align*}
\end{proof}

\begin{lem}
The evolution of curvature $\kappa$ is given by
\begin{align}\label{k'}
\frac{\partial \kappa}{\partial t}=-\frac{\partial^{2} F}{\partial s^{2}}-\left(\kappa^{2}+K\right) F.
\end{align}
\end{lem}
\begin{proof}

We have
\begin{align*}
\frac{\partial}{\partial t} \kappa &=\frac{\partial}{\partial t}\left\langle\overline{\nabla}_{s} T, -\nu\right\rangle \\
&=\left\langle\overline{\nabla}_{t} \overline{\nabla}_{s} T, -\nu\right\rangle+\left\langle\overline{\nabla}_{s} T, -\overline{\nabla}_{t} \nu\right\rangle \\
&=\left\langle\overline{\nabla}_{s} \overline{\nabla}_{t} T-F \kappa \overline{\nabla}_{s} T+ F \bar{R}(T, \nu) T, -\nu\right\rangle+\left\langle-\kappa \nu, \overline{\nabla}_{s} F T\right\rangle\\
=&\left\langle\overline{\nabla}_{s}\left(\overline{\nabla}_{s} F \nu\right) +F \kappa^{2} \nu, -\nu\right\rangle-K F \\
=& -\frac{\partial^{2} F}{\partial s^{2}}-\left(\kappa^{2}+K\right) F.
\end{align*}

where $K=\langle \bar{R}(T, \nu)T, \nu\rangle$ is the Gauss curvature of $M^2(K)$.
\end{proof}

\section{On the flow \texorpdfstring{\eqref{flow}}{}}
We collect some useful facts on $ M^2(K)$ here.
We recall the Minkowski formula on $M^2$:
\begin{align}\label{mink}
\int_{\partial \Omega} \phi'(r)=\int_{\partial \Omega}\kappa\langle V, \nu\rangle
\,.
\end{align}
We also have Brendle's Heintze-Karcher type inequality \cite{brendle2013constant}:
\begin{align}\label{brendle}
\int_{\partial \Omega}\frac{\phi'(r) }{\kappa}\ge \int_{ \Omega} 2\phi'(r) dA=\int_{\partial \Omega}\langle V, \nu\rangle.
\end{align}
Note that it was proved in \cite{brendle2013constant} that equality holds only if $\partial \Omega$ is umbilical, which is automatic when $n=2$.
On the other hand, by \cite[Theorem 1]{ros1987compact} and \cite{qiu2015generalization}, the equality holds if and only if $\Omega$ is a geodesic ball for all dimensions.

Now we consider the locally constrained curve flow \eqref{flow}.
Let $F=\frac{\phi'(r)}{\kappa}-u$ be the flow speed.
\begin{prop}\label{prop L}
Under the flow \eqref{flow}, the length $L$ of $\gamma_t$ is preserved and the area $A$ enclosed by $\gamma_t$ satisfies
\begin{equation}\label{A'}
A^{\prime}(t)=\int_{\gamma_t}\left(\frac{\phi'(r)}{\kappa}-\langle V, \nu\rangle\right) \ge 0.
\end{equation}
The isoperimetric deficit $L^2-4\pi A+ K A^2$ is also monotone decreasing.
\end{prop}
\begin{proof}
By \eqref{L'} and \eqref{mink},
\begin{align*}
L^{\prime}(t)=\int_{\gamma_t}\left(\phi'(r)-\kappa \langle V, \nu\rangle\right) ds=0.
\end{align*}
By the coarea formula and \eqref{brendle},
\begin{align*}
A^{\prime}(t)=\int_{\gamma_t}\langle \partial_tX, \nu\rangle = \int_{\gamma_t}Fds=\int_{\gamma_t}\left(\frac{\phi'(r)}{\kappa}-\langle V, \nu \rangle \right) ds\ge 0.
\end{align*}

If $K=1$, we will see in Proposition \ref{prop rho} that $X(\cdot, t)$ will stay inside the hemisphere centered at $0$, and in particular the area enclosed by it is less than $2\pi$.
From the above,
\begin{align*}
\frac{d }{d t}\left(L^2-4\pi A+ K A^2\right)
=2LL'-4\pi A'+2KAA'\le0.
\end{align*}
\end{proof}

\begin{prop}\label{prop k}
Suppose $X_0$ is strictly convex, then $\gamma_t$ remains strictly convex along the flow \eqref{flow}. Furthermore we have the uniform two-sided estimate
\[
\min \kappa_0\le \kappa(\cdot, t)\le \max \kappa_0\,.
\]
\end{prop}
\begin{proof}
Let $s$ be the arclength parameter.
By \eqref{k'},
\begin{align}\label{kt}
\frac{\partial \kappa}{\partial t}=\phi' \kappa^{-2} \frac{\partial^2\kappa}{\partial s^2}-2 \phi' \kappa^{-3} \left(\frac{\partial \kappa}{\partial s}\right)^{2}+2 \kappa^{-2} \frac{\partial \phi'}{\partial s} \frac{\partial \kappa}{\partial s} -\frac{1}{\kappa } \frac{\partial^2 \phi'}{\partial s^2} +\frac{\partial^2 u}{\partial s^2}-\left(\kappa^{2}+ K \right)F
\,.
\end{align}
We compute
$\frac{\partial \phi'}{\partial s}=\langle\overline{\nabla} \phi', T\rangle=\left\langle\phi^{\prime \prime}(r) \partial_{r}, T\right\rangle=\frac{\phi^{\prime \prime}(r)}{\phi(r)}\left\langle\phi(r) \partial_{r}, T\right\rangle=-K\langle V, T\rangle$
and
\begin{equation}\label{V}
\overline \nabla^2 \phi'=- K \overline \nabla V= -K \phi'(r)\; \bar g.
\end{equation}

So by \eqref{V},
\begin{align}\label{fss}
\frac{\partial^2 \phi'}{\partial s^2}=- K \left\langle\overline {\nabla}_{T} V, T\right\rangle- K \left\langle V, \overline {\nabla}_{T} T\right\rangle=- K \phi'\bar g\langle T, T\rangle+ K \kappa\langle V, \nu\rangle
=- K \left(\phi'-\kappa u\right).
\end{align}
We compute $\frac{\partial u}{\partial s}=\langle V, \kappa T\rangle $ and
\begin{align}\label{uss}
\frac{\partial^2 u}{\partial s^2}
=\frac{\partial }{\partial s}\left(\kappa\langle V, T\rangle\right)=\frac{\partial \kappa}{\partial s}\langle V, T\rangle+\kappa \phi'-\kappa^{2} u=\frac{\partial \kappa}{\partial s}\langle V, T\rangle+\kappa(\phi'-\kappa u).
\end{align}
Putting \eqref{fss} and \eqref{uss} into \eqref{kt},
we get
\begin{align*}
\frac{\partial \kappa}{\partial t}
=&\phi' \kappa^{-2} \frac{\partial^2\kappa}{\partial s^2}+\frac{\partial \kappa}{\partial s}\left(2 \kappa^{-2} \frac{\partial \phi'}{\partial s}-2 \phi' \kappa^{-3} \frac{\partial \kappa}{\partial s}+\langle V, T\rangle\right)\\
=&\phi' \kappa^{-2} \frac{\partial^2\kappa}{\partial s^2}+\frac{\partial \kappa}{\partial s}\left(2 \kappa^{-1} \frac{\partial }{\partial s}\left(\frac{\phi'}{\kappa}\right)+\langle V, T\rangle\right).
\end{align*}
Again, in the case where $K=1$, we will prove in Proposition \ref{prop rho} that $\gamma_t$ lies in the hemisphere centered at $0$ and so $\phi'(r)>0$.

By the maximum principle, $\kappa$ remains positive and uniformly bounded by its initial values.
\end{proof}


Every star-shaped curve can be regarded as a graph of $\rho: \mathbb S^1 \rightarrow(0, \infty)$ over $\mathbb S^1$.
For such a graph,
\begin{align*}
\nu=\frac{1}{\sqrt{1+\rho_\theta^{2} \phi(\rho)^{-2}}}\left(\partial_{r}-\rho_\theta\phi(\rho)^{-2} \, \partial_{\theta}\right)
\end{align*}
and
\begin{align}\label{k}
\kappa=\frac{\phi (\rho)^2 \phi' (\rho)+2 \rho_\theta^{ 2} \phi'(\rho)-\rho_{\theta\theta} \phi(\rho)}{\left(\phi(\rho)^2+\rho_\theta^2 \right)^{\frac{3}{2}}}.
\end{align}

For a family of curves that can be represented as graphs over $\mathbb S^1$, we can write
\begin{align*}
X(x, t)=\left(\rho(\psi(x, t), t), \psi(x, t)\right)
\end{align*}
where $\psi(\cdot, t): \mathbb S^1 \rightarrow \mathbb S^1$ is a family of diffeomorphism and $\rho(\cdot, t): \mathbb S^1 \rightarrow(0, \infty)$ is the radial distance.

By direct computation,
\begin{align*}
\frac{\partial X}{\partial t}=\left(\frac{\partial \rho}{\partial \theta}\frac{\partial \psi}{\partial t}+\frac{\partial \rho}{\partial t}\right)\frac{\partial }{\partial r}+\frac{\partial \psi}{\partial t}\frac{\partial }{\partial \theta}.
\end{align*}
So
\begin{align*}
\left\langle\frac{\partial X}{\partial t}, \nu\right\rangle
=& \frac{1}{\sqrt{1+\rho_\theta^{2} \phi(\rho)^{-2}}}\left[\left(\frac{\partial \rho}{\partial \theta} \frac{\partial \psi}{\partial t}+\frac{\partial \rho}{\partial t}\right)+ \phi (\rho)^2\left(\frac{\partial \rho}{\partial \theta} \phi(\rho)^{-2}\frac{\partial \psi}{\partial t}\right)\right]\\
=& \frac{1}{\sqrt{1+\rho_\theta^{2} \phi(\rho)^{-2}}}\frac{\partial \rho}{\partial t}.
\end{align*}
Therefore we can express the flow \eqref{flow} equivalently as
\begin{align*}
\frac{1}{\sqrt{1+\rho_\theta^{2} \phi (\rho)^{-2}}}\frac{\partial \rho}{\partial t}=
\frac{\phi'(\rho)}{\kappa}-\langle V, \nu\rangle=
\frac{\phi'(\rho)}{\kappa}-\frac{ \phi(\rho)}{\sqrt{1+\rho_\theta^{2} \phi(\rho)^{-2}}}.
\end{align*}
So let us consider
\begin{equation}
\label{rho}
\begin{cases}
&\dfrac{\partial \rho }{\partial t}
=\dfrac{ \phi'(\rho) \sqrt{1+\rho_\theta^2 \phi(\rho)^{-2}}}{\kappa}-\phi(\rho)\\
&\rho(x, 0)=\rho_0(x).
\end{cases}
\end{equation}
where we regard $\kappa$ as a function of $\rho$, $\rho_\theta$ and $\rho_{\theta\theta}$ given by \eqref{k}.

\begin{prop}\label{prop rho}
If $\rho$ is a solution to \eqref{rho}, then
$\min\rho(\cdot, 0)\le\rho \le \max \rho(\cdot, 0)$.
In particular, if $K=1$ and the initial curve $X_0$ lies in the hemisphere centered at $0$, then the solution $X(\cdot, t)$ to \eqref{flow} always lies in the same hemisphere.
\end{prop}
\begin{proof}
At a spatial maximum point of $\rho, \nabla \rho=0, $ so \eqref{rho} becomes
\begin{align*}
\frac{\partial \rho}{\partial t}=\frac{\phi'(\rho)}{\kappa}-\phi(\rho).
\end{align*}

By comparison principle, as the curve is contained in the circle with radius $\rho $ centered at $0$, the curvature at this point must be greater than or equal to $ \frac{\phi'(\rho)}{\phi(\rho)} $, which is the geodesic curvature of the circle with radius $\rho$, hence $\frac{\partial \rho}{\partial t} \le 0$ at this point. By a result of Hamilton \cite{ham}, we conclude that $\rho \le \max \rho(\cdot, 0)$. Similarly, at a spatial minimum, $\rho \ge \min \rho(\cdot, 0)$.
\end{proof}

\begin{prop}\label{prop rho'}
Along the flow \eqref{rho},
$\displaystyle \left|\frac{\partial \rho}{\partial \theta}\right| \le C\max_{ \mathbb S^{1}}\left| \frac{\partial \rho_{0}}{\partial \theta}\right|$, where $\displaystyle C$ is a constant which depends continuously on $\displaystyle \max_{\mathbb S^1}\rho_0$ and $\displaystyle \min_{\mathbb S^1}\rho_0$ only.
\end{prop}
\begin{proof}
Let $h(\tau)=2 \mathrm{ct}_K^{-1}\left(e^{-\tau}\right)$ for $\tau<0$, where $\mathrm{ct}_K^{-1}$ is the inverse function of $x\mapsto \frac{\phi'(x)}{\phi(x)}$.
Direct computation shows that $h'(\tau)=\phi (h(\tau))$.

Now, let $\rho(\theta, t)=h(\tau(\theta, t))$, we compute
\begin{align*}
{\rho_{\theta}} =h^{\prime}(\tau) \tau_\theta
=\phi (h(\tau)) \tau_\theta= \phi(\rho)\tau_\theta
\end{align*}
and
\begin{align*}
\rho_{\theta\theta} =\phi'(h(\tau)) \phi (h (\tau))\tau_{\theta}^2+ \phi (h(\tau)) {\tau_{\theta\theta}}
=\phi'(\rho) \phi (\rho)\tau_{\theta}^2+ \phi (\rho) \tau_{\theta\theta}.
\end{align*}
In view of \eqref{k}, \eqref{rho} then becomes
\begin{align}\label{rt}
\frac{\partial \tau}{\partial t}
=\frac{\phi' (\rho)\left(1+\tau_\theta^{2}\right)^{2}}{\phi' (\rho) \left(1+\tau_\theta^{2} \right)-\tau_{\theta\theta}}-1
\end{align}
where
\[
\phi' (\rho)=
\begin{cases}
-\coth \tau\; \textrm{if }K=-1\\
- \tanh\tau \;\textrm{if }K=1
\end{cases}
\,.
\]
We regard the RHS of \eqref{rt} as $G(\phi'(\rho(\tau)), \tau_\theta, \tau_{\theta\theta})=G(v, \tau_\theta, \tau_{\theta\theta})$.
Direct computation gives
$\frac{\partial G}{\partial v}=-\frac{\left(1+{\tau_{\theta}}^2\right)^{2} {\tau_{\theta\theta}}}{ (v(1+{\tau_\theta}^{2}) -{\tau_{\theta\theta}})^2 }$.
Let $w:=\frac{1}{2}{\tau_\theta}^2$. Then (by slightly abusing notation)
\begin{align*}
\frac{\partial w}{\partial t}
=&{\tau_{\theta}} \frac{\partial }{\partial \theta} \left({\tau}_t\right)\\
=&{\tau_{\theta}}\left(\frac{\partial G}{\partial v} {v_{\theta}}+\frac{\partial G}{\partial {\tau_{\theta}}} {\tau_{\theta\theta}}+\frac{\partial G}{\partial {\tau_{\theta\theta}}} {\tau_{\theta\theta\theta}} \right)\\
=&{\tau_{\theta}} \left(-\frac{\left(1+\tau_\theta^{2}\right)^{2} {\tau_{\theta\theta}}}{\left(\phi' (\rho)\left(1+{\tau_\theta}^{2}\right)-{\tau_{\theta\theta}}\right)^{2}} \phi''(\rho)\phi(\rho)\tau_\theta+\frac{\partial G}{\partial {\tau_{\theta}}}{\tau_{\theta\theta}} +\frac{\partial G}{\partial {\tau_{\theta\theta}}} {\tau_{\theta\theta\theta}}\right) \\
=& \left(-\frac{{\tau_{\theta}}\phi^{\prime \prime}(\rho) \phi(\rho) \left(1+\tau_\theta^{ 2}\right)^{2} }{\left(\phi' (\rho)\left(1+{\tau_\theta}^{2}\right)-{\tau_{\theta\theta}}\right)^{2}} +\frac{\partial G}{\partial {\tau_{\theta}}}\right) w_\theta +\frac{\partial G}{\partial {\tau_{\theta\theta}}} {\tau_{\theta\theta\theta}} \tau_\theta\\
=& \left(-\frac{{\tau_{\theta}}\phi^{\prime \prime}(\rho) \phi(\rho) \left(1+\tau_\theta^{ 2}\right)^{2} }{\left(\phi' (\rho)\left(1+{\tau_\theta}^{2}\right)-{\tau_{\theta\theta}}\right)^{2}} +\frac{\partial G}{\partial {\tau_{\theta}}}\right) w_\theta +\frac{\partial G}{\partial {\tau_{\theta\theta}}} \left(w_{\theta\theta}-{\tau_{\theta\theta}}^2\right).
\end{align*}

It is not hard to see that
$\frac{\partial G}{\partial {\tau_{\theta\theta}}}= \frac{\left(1+{\tau_\theta}^{2}\right)^{2} \phi' (\rho) }{(\phi' (\rho) (1+ {\tau_\theta}^{2}) - \tau_{\theta\theta})^{2}}
>0$ ($\rho\in(0, \frac{\pi}{2})$ if $K=1$ by Proposition \ref{prop rho}), so at a spatial maximum point of $w$, $\frac{\partial w}{\partial t}\le 0$ and hence by \cite{ham},
$| \frac{\partial \tau}{\partial \theta}|^{2} \le \max_{ \mathbb S^{1}}\left|\frac{\partial \tau_{0}}{\partial \theta}\right|^{2}$.

Note that $\tau\mapsto 2\mathrm{ct}_K^{-1}(e^{-\tau}):(-\infty, 0)\to (0,
\infty)$ is either $\tau\mapsto e^\tau$ ($K=0$), $\tau\mapsto\text{arccoth } e^{-\tau}$ ($K=-1$) or $\tau\mapsto \text{arccot } e^{-\tau}$.
In all of these cases it is a strictly increasing convex function, and the values of $\rho$ and $\tau$
both lie in a compact interval, which can be determined by $\max \rho_0$ and
$\min \rho_0$.
Therefore
$\left|\frac{\partial \rho}{\partial \theta}\right| \le C \max_{\mathbb S^{1}}\left|\frac{\partial \rho_{0}}{\partial \theta}\right|$ where $C$ is determined by $\max\rho_0$ and $\min \rho_0$ only.
\end{proof}

\begin{prop}
If $X_0$ contains the origin, then $X(\cdot, t)$ contains the origin along the flow \eqref{flow}.
Furthermore we have the two-sided estimate
\[
C_1\ge u\ge C_2>0
\]
where the constants $C_1$ and $C_2$ depend on
the maxima and minima of $\left|\frac{\partial \rho_{0}}{\partial \theta}\right|$ and $\rho_{0}$ only.
\end{prop}
\begin{proof}
Clearly $u\le \phi(\rho) \le C_1$ by Proposition \ref{prop rho}.
It follows from Proposition \ref{prop rho'} that $\rho_\theta$ does not blow up along the flow. The expression $u=
\frac{\phi(\rho)}{\sqrt{1+\rho_{\theta}^{2} \phi(\rho)^{-2}}}$
then implies that $u\ge C>0$ where the constant $C$ depends on the maxima and
	minima of $\left|\frac{\partial \rho_{0}}{\partial \theta}\right|$ and
	$\rho_{0}$ only.
\end{proof}

\section{Long time existence and convergence}
\subsection{Long time existence}\label{sec. LTE}

By \eqref{k}, Proposition \ref{prop k}, and Proposition \ref{prop rho'}, we conclude that
if $\rho$ is a positive, admissible solution of \eqref{rho} on $\mathbb S^1 \times[0, T) $, then for any $t \in[0, T)$ we have
\begin{align}\label{c2}
\left\|\rho\right\|_{ C^{2}\left(\mathbb S^{1} \times[0, T)\right)} \le C
\end{align}
where $ C$ depends continuously only on $ \max_{\mathbb S^1} \rho_{0},
\min_{\mathbb S^1} \rho_{0}, \max_{\mathbb S^1}\left|\frac{\partial
\rho_{0}}{\partial \theta}\right|$, $\min_{\mathbb S^1}\left|\frac{\partial
\rho_{0}}{\partial \theta}\right|$, $ \max_{\mathbb S^1}\left|\frac{
\partial^2\rho_{0}}{\partial \theta^2}\right|$ and $ \min_{\mathbb
S^1}\left|\frac{ \partial^2\rho_{0}}{\partial \theta^2}\right|$. Since
$\displaystyle \kappa$ is uniformly bounded from both above and below by
positive constants, by applying results of Krylov and Safonov (\cite[Section
5.5]{krylov1987nonlinear}), we can obtain $\left\|\rho\right\|_{C^{2, \alpha}}\le C$ where
$C$ depends on $\rho_0$ only.
By a standard argument, we can derive uniform $C^k$ estimate of $\rho$ for all $k\ge 2$ and then conclude that the flow \eqref{flow} exists for all time $t \in 0, \infty)$.

\subsection{Smooth convergence}
First, we observe the following weak subconvergence.

\begin{lem}\label{lem subseq}
There exists a sequence of times $\{t_j\}$, $t_j\rightarrow\infty$, such that $Q(t_j) \searrow 0$, where
\[
Q(t) = \int_{\gamma_t} F\, ds = \int_{\gamma_t} \Big(\frac{\phi'}{\kappa} - u\Big)\, ds
\,.
\]
Along this sequence of times, $\gamma_t$ converges smoothly to a geodesic circle
centered at the origin.
\end{lem}
\begin{proof}
Since $Q = A'$, we have
\[
\int_0^t Q(\tau)\, d\tau = A(t) - A(0)
\,.
\]
Proposition \ref{prop rho} implies that the RHS is uniformly bounded by $2\pi \Phi(\max_{\mathbb S^1}\rho_0)$.
Then, Brendle's inequality yields $Q(t) \ge 0$ so that $Q$ is a non-negative
function in $L^1((0, \infty))$. This yields the existence of the sequence $t_j$.

For this sequence of times $t_j \rightarrow \infty, $ the $C^0$ and uniform regularity estimates imply that there exists a subsequence, still denoted by $t_j, $ such that $\gamma_{t_j}$ converges smoothly to a limit $\gamma_\infty, $ which satisfies $\int_{\gamma_\infty}\left(\frac{\phi'}{\kappa}-u\right)ds=0$. By the rigidity of Heintze-Karcher inequality \cite{ros1987compact, qiu2015generalization}, the limit curve $ \gamma_\infty$ must be a geodesic circle. That is, we have a subsequence which converges smoothly to a geodesic circle.

We now claim that the geodesic circle $\gamma_\infty$ is centered at $0$. Suppose the radius of $\gamma_\infty$ is $\rho_\infty$, such that its curvature is $ \frac{\phi'(\rho_\infty)}{\phi(\rho_\infty)}$. This circle is a stationary solution to the flow \eqref{flow} with length $2\pi \phi(\rho_\infty)$. If this circle is not centered at $0$, then at the maximum point of $\rho$, by \eqref{rho}, we have $\frac{\partial \rho}{\partial t}=\frac{\phi' (\rho) }{\kappa}-\phi (\rho)=
\frac{ \phi' (\rho) \phi (\rho_\infty)}{\phi' (\rho_\infty)}- \phi(\rho) =\frac{\phi(\rho_\infty-\rho)}{\phi'(\rho_\infty)}<0$. This contradicts the fact that $\gamma_\infty$ is stationary.
\end{proof}

\begin{prop}
There exists $\rho_\infty>0$ such that along the flow \eqref{rho}, $ \left\|\rho(\cdot, t)-\rho_\infty\right\|_{C^k(\mathbb S^1)}\to0$ as $t\to\infty$ for all $k\ge 0$.
\end{prop}
\begin{proof}
By Lemma \ref{lem subseq}, there exists $t_j\to \infty$ and $\rho_\infty$ such that
\begin{equation}\label{s4.2-1}
\left\|\rho(\cdot, t_j)-\rho_\infty\right\|_{C^k(\mathbb S^1)}\to0
\end{equation}
for all $k\ge 0$. On the other hand, by the proof of Proposition \ref{prop rho}, $\rho_{\max}(t)=\max_{\theta\in \mathbb{S}^1}\rho(\theta, t)$ is non-increasing in time and $\rho_{\min}(t)=\min_{\theta\in \mathbb{S}^1}\rho(\theta, t)$ is non-decreasing in time. This together with subconvergence \eqref{s4.2-1} implies that
\begin{equation}\label{s4.2-C0}
\left\|\rho(\cdot, t)-\rho_\infty\right\|_{C^0(\mathbb{S}^1)}~\to~0
\end{equation}
as $t\to\infty$. Recall a special case of the Gagliardo-Nirenberg interpolation inequality \cite{nirenberg1966extended} (see \cite{Aub98} for a statement of the inequality on Riemannian manifolds): for positive integers $j, k $ satisfying ${j}<{k}$, there exists a constant $C$ depending only on $j, k$ such that
\begin{equation}\label{s4.2-GN}
\left\|\frac{\partial^j}{\partial \theta^j}f\right\|_{L^2(\mathbb{S}^1)}\le C \left\|\frac{\partial^k}{\partial \theta^k}f\right\|_{L^2(\mathbb{S}^1)}^{\delta}\left\|f\right\|_{L^2(\mathbb{S}^1)}^{1-\delta},
\end{equation}
for all function $f\in W^{k, 2}(\mathbb{S}^1)$, where $\delta=j/k$. Choosing $f=\rho(\theta, t)-\rho_\infty$ in \eqref{s4.2-GN}, the uniform $C^k$ estimate of $\rho$ derived in Section \ref{sec. LTE} and the $C^0$ convergence \eqref{s4.2-C0} imply that for each positive integer $j$ and any $\varepsilon>0$, $\|\frac{\partial^j}{\partial \theta^j}\rho(\cdot, t)\|_{L^2(\mathbb S^1)}<\varepsilon$ holds for large enough $t$. Then Sobolev embedding theorem implies that $\|\rho(\cdot, t)-\rho_\infty\|_{C^{j-1}(\mathbb S^1)}<\varepsilon$ provided time $t$ is taken sufficiently large.
This means that $\rho(\cdot, t)$ converges smoothly to $\rho_\infty$ as $t\to\infty$.
\end{proof}
\subsection{Exponential convergence}
The exponential convergence can be obtained by studying the linearisation of the flow \eqref{flow}.
\begin{lem}
The linearisation of \eqref{rho} along the stationary solution $\rho=\rho_\infty$ is
\begin{equation}\label{s4.3-4}
\frac{\partial}{\partial t}h=\frac{1}{\phi'(\rho_\infty)}h_{\theta\theta}.
\end{equation}
\end{lem}
\begin{proof}
We write
\begin{equation*}
\rho(\theta, t)=\rho_\infty+\varepsilon h(\theta, t)
\end{equation*}
for small $\varepsilon$. We have
\begin{align*}
\frac{d}{d\varepsilon}\bigg|_{\varepsilon=0}\phi'(\rho) =& \phi''(\rho_\infty)h \\
\frac{d}{d\varepsilon}\bigg|_{\varepsilon=0} u=& \phi'(\rho_\infty) h \\
\frac{d}{d\varepsilon}\bigg|_{\varepsilon=0} \kappa=& \frac{1}{\phi^3(\rho_\infty)}\left(-\phi(\rho_\infty)h_{\theta\theta}+2\phi(\rho_\infty)(\phi'(\rho_\infty))^2h+\phi(\rho_\infty)^2\phi''(\rho_\infty)h\right)\\
&\quad -\frac{3}{2}\frac{\phi^2(\rho_\infty)\phi'(\rho_\infty)}{\phi^5(\rho_\infty)}2\phi(\rho_\infty)\phi'(\rho_\infty)h\\
=& -\frac{1}{\phi^2(\rho_\infty)}h_{\theta\theta}+\frac{1}{\phi^2(\rho_\infty)}\left(\phi(\rho_\infty)\phi''(\rho_\infty)-(\phi'(\rho_\infty))^2\right)h\\
\frac{d}{d\varepsilon}\bigg|_{\varepsilon=0}&\sqrt{1+\frac{(\rho_\theta)^2}{\phi^2(\rho)}}=0.
\end{align*}
Using the above equations, we obtain \eqref{s4.3-4}.
\end{proof}

\begin{proof}[Proof of Theorem \ref{s4.thm}]
We have proved that $\gamma_t$ converges smoothly to a geodesic circle $\gamma_\infty$ of radius $\rho_\infty$ centered at $0$ as $t\to\infty$. So for sufficiently large time $t$, $\gamma_t$ lies in a small neighborhood of $\gamma_\infty$. Denote
\begin{equation*}
\sigma(\theta, t)= \rho(\theta, t)-\rho_\infty
\end{equation*}
which converges to zero smoothly as $t\to\infty$. Let $G[\rho(\theta, t)]$ denote the right-hand side of \eqref{rho}. Then using the linearisation \eqref{s4.3-4} we can write the flow \eqref{rho} as
\begin{align*}
\frac{\partial}{\partial t}\sigma(\theta, t)=&\frac{\partial}{\partial t}(\rho(\theta, t)-\rho_\infty)\\
=& G[\rho(\theta, t)]\\
=& G[\rho_\infty]+DG\bigg|_{\rho_\infty}(\sigma(\theta, t))+O(\|\sigma(\theta, t)\|^2_{C^2(\mathbb{S}^1)})\\
=& \frac{1}{\phi'(\rho_\infty)}\frac{\partial^2}{\partial\theta^2}\sigma(\theta, t)+ \eta
\end{align*}
where $\eta(\theta, t)$ denotes the error terms which can be controlled by $\|\sigma(\theta, t)\|_{C^2(\mathbb{S}^1)}^2$.

This implies
\begin{equation}\label{sigma}
\begin{split}
\frac{d}{d t} \int_{\mathbb S^1} \sigma(\theta, t)^2 d \theta \le & \frac{1}{\phi^{\prime}\left(\rho_\infty\right)} \int_{\mathbb S^1} \sigma(\theta, t) \frac{\partial^{2}}{\partial \theta^{2}}\left(\sigma(\theta, t)\right)
+C \int_{\mathbb S^1}\left|\sigma(\theta, t)\right|\left\|\sigma(\theta, t)\right\|_{C^{2}(\mathbb S^1)}^{2} \\
=&-\frac{1}{\phi^{\prime}\left(\rho_\infty\right)} \int_{\mathbb S^1}\left(\frac{\partial}{\partial \theta}\left(\sigma(\theta, t)\right)\right)^{2}
+C \int_{\mathbb S^1}\left|\sigma(\theta, t)\right|\left\|\sigma(\theta, t)\right\|_{C^{2}(\mathbb S^1)}^{2} \\
\le &-\frac{ 1 }{ 4 \phi^{\prime}\left(\rho_\infty\right)} \int_{\mathbb S^1} \sigma(\theta, t)^{2}
+C \int_{\mathbb S^1}\left|\sigma(\theta, t)\right|\left\|\sigma(\theta, t)\right\|_{C^{2}(\mathbb S^1)}^{2}.
\end{split}
\end{equation}
In the last line, we have used the Wirtinger inequality
$\int_{0}^{2\pi} f(\theta)^2d\theta\le 4\int_0^{2\pi}f'(\theta)^2d\theta$ if
$f\in C^1(\mathbb S^1)$ attains zero somewhere on its domain (\cite[Theorem
257]{hardy1952inequalities}). The reason $\sigma$ attains $0$ is as follows. If
$\sigma(\theta, t)<0$ for all $\theta$, then as $\gamma_t$ is convex and
contains the origin (which we recall is the centre of $\gamma_\infty$), it is
enclosed by $\gamma_\infty$ and has a shorter length than $\gamma_\infty$. But since
the length is preserved along the flow, this is a contradiction. Similarly it
is not possible to have $\sigma(\theta, t)>0$ for all $\theta$.

We estimate the last term on the right-hand side of \eqref{sigma}.
Applying the Gagliardo-Nirenberg interpolation inequality \eqref{s4.2-GN}, for positive integers $j, k $ satisfying ${j}<{k}$, there exists a constant $C$ depending only on $j, k$ such that
\begin{equation}\label{s4.3-GN}
\left\|\frac{\partial^j}{\partial \theta^j}\sigma(\theta, t)\right\|_{L^2(\mathbb{S}^1)}\le C \left\|\frac{\partial^k}{\partial \theta^k}\sigma(\theta, t)\right\|_{L^2(\mathbb{S}^1)}^{\delta}\left\|\sigma(\theta, t) \right\|_{L^2(\mathbb{S}^1)}^{1-\delta},
\end{equation}
where $\delta=j/k$. By the Sobolev embedding theorem, $\|\sigma(\theta,
	t)\|_{C^2(\mathbb{S}^1)} \le C \|\sigma(\theta, t)\|_{W^{j,
	2}(\mathbb{S}^1)}$ for $j\ge 3$. Choosing $k$ larger in
	\eqref{s4.3-GN}, the $C^\infty$ estimate implies that for any
	$0<\varepsilon<1$
\begin{align}\label{s4.3-9}
\|\sigma(\theta, t)\|_{C^2(\mathbb{S}^1)}^2\le & C(\varepsilon) ||\sigma(\theta, t)||_{L^2(\mathbb{S}^1)}^{1+\varepsilon}
\,.
\end{align}
We can choose $\varepsilon=1/2$ (this also fixes the $k$ in \eqref{s4.3-GN}), and then apply the H\"{o}lder inequality and \eqref{s4.3-9} to the last term of \eqref{sigma}.
We obtain
\begin{align}\label{s4.3-10}
\int_{\mathbb{S}^1} \left|\sigma(\theta, t)\right| \left\|\sigma(\theta, t)\right\|_{C^2(\mathbb{S}^1)}^2 \le & \left\|\sigma(\theta, t)\right\|_{L^2(\mathbb{S}^1)}\left(\int_{\mathbb{S}^1} \left\|\sigma(\theta, t)\right\|_{C^2(\mathbb{S}^1)}^4d\theta\right)^{1/2}\nonumber\\
\le & C \left(\int_{\mathbb{S}^1}\sigma(\theta, t)^2d\theta\right)^{1+\frac{1}{4}}.
\end{align}
Substituting
\eqref{s4.3-10} into \eqref{sigma} implies
\begin{equation}\label{s4.3-11}
\begin{split}
\frac{d}{d t} \int_{\mathbb{S}^1} \sigma(\theta, t)^{2} d \theta
\le &-\frac{1}{4\phi^{\prime}\left(\rho_{\infty}\right)} \int_{\mathbb{S}^1} \sigma(\theta, t)^{2} d \theta
+C\left(\int_{\mathbb{S}^1} \sigma(\theta, t)^{2} d \theta\right)^{1+\frac{1}{4}}\\
=&-\frac{1}{4\phi^{\prime}\left(\rho_{\infty}\right)} \int_{\mathbb{S}^1} \sigma(\theta, t)^{2} d \theta\left(1 - C'\left(\int_{\mathbb{S}^1} \sigma(\theta, t)^{2} d \theta\right)^{\frac{1}{4}}\right) \\
\le &-\frac{1}{4\phi^{\prime}\left(\rho_{\infty}\right)} \int_{\mathbb{S}^1} \sigma(\theta, t)^{2} d \theta(1-\delta)
\,.
\end{split}\end{equation}
Here for any $\delta>0$ there exists a time $t_\delta$ such that the above
estimate holds for all $t>t_\delta$.
This is because we already have $\int_{\mathbb{S}^1}\sigma(\theta,
t)^2d\theta\to 0$.
Choosing $\delta=1/2$, \eqref{s4.3-11} implies
\begin{equation}\label{s4.3-12}
\int_{\mathbb{S}^1}\sigma(\theta, t)^2d\theta\le C e^{-\frac{t}{8\phi'(\rho_\infty)}}.
\end{equation}
Finally, applying again the interpolation inequality \eqref{s4.3-GN} and Sobolev embedding theorem, we can obtain the exponential convergence of $\rho(\theta, t)$ to $\rho_\infty$ in any $C^k$ topology:
\begin{equation*}
\|\sigma(\theta, t)\|_{C^k(\mathbb{S}^1)}\le C_1 e^{-C_2t}
\end{equation*}
for each $k\ge 0$, where $C_1$ and $C_2$ depends on $\rho_0$ and $k$ only. This completes the proof of Theorem \ref{s4.thm}.
\end{proof}

\section{Applications}

\subsection{Isoperimetric inequality}

We first provide a new proof of the isoperimetric inequality for convex curves.
\begin{thm}\label{s5.thm1}
Let $\gamma$ be a smooth, closed, convex curve in $M^2(K)$ with $K=-1$, $1$ and $0$. Then the isoperimetric inequality
\begin{equation}\label{s5.1-1}
L^2\ge 4\pi A-KA^2
\end{equation}
holds for $\gamma$. Equality holds in \eqref{s5.1-1} if and only if $\gamma$ is a geodesic circle.
\end{thm}
\begin{proof}
We divide the proof into two steps.

\textbf{\underline{Step 1: }} We first prove \eqref{s5.1-1} for smooth,
	strictly convex (i.e., $\kappa>0$) closed curve $\gamma$ which encloses
	a bounded domain $\Omega$. We choose a point $o$ in the interior of
	$\Omega$ such that $\gamma$ is star-shaped with respect to $o$ and can
	be expressed as a graph of radial function $\rho_0:\mathbb{S}^1\to
	\mathbb{R}_+$. Let $X_0=(\rho_0(\theta), \theta): \mathbb{S}^1\to
	M^2(K)$ be the embedding of the curve $\gamma$. We evolve $\gamma$
	along the flow \eqref{flow}.

By Proposition \ref{prop L}, the isoperimetric deficit $Q(\gamma_t)=L^2-4\pi A+K A^2$ is non-increasing in time along the flow \eqref{flow}. In previous section, we have proved that for initially strictly convex curve, the flow \eqref{flow} converges smoothly and exponentially to a geodesic circle $S_{\rho_\infty}$ of radius $\rho_\infty$. This implies that
\begin{equation*}
Q(\gamma)\ge \lim_{t\to\infty} Q(\gamma_t)=Q(S_{\rho_\infty})=0
\end{equation*}
for strictly convex curve $\gamma$.

If equality \eqref{s5.1-1} holds for some strictly convex closed curve $\gamma$, then $Q(\gamma)=Q(\gamma_t)$ for any positive time $t$. By the proof of Proposition \ref{prop L}, each curve $\gamma_t$ is a geodesic circle. In particular, $\gamma$ is a geodesic circle.

\textbf{\underline{Step 2: }} Now for a smooth, closed and convex curve $\gamma$, let $X_0: \mathbb{S}^1\to M^2(K)$ be the embedding of the curve $\gamma$. We evolve $\gamma$ along the curve shortening flow
\begin{equation}\label{s5.1-csf}
\frac{\partial }{\partial t}X=-\kappa \nu.
\end{equation}
By the general evolution equation \eqref{k'}, the curvature $\kappa$ of the solution $\gamma_t=X(\mathbb{S}^1, t)$ of \eqref{s5.1-csf} satisfies
\begin{equation}\label{s5.1-csf2}
\frac{\partial }{\partial t}\kappa=\frac{\partial^2}{\partial s^2}\kappa+\kappa(\kappa^2+K).
\end{equation}
Since $\kappa\ge 0$ on the initial curve $\gamma$ and there exists at least one point on $\gamma$ with $\kappa>0$, applying the strong maximum principle to \eqref{s5.1-csf2} implies that $\gamma_t$ is strictly convex for $t>0$. By \textbf{\underline{Step 1}}, the inequality \eqref{s5.1-1} holds for each $\gamma_t$. Let $t\to 0$, we obtain that the inequality \eqref{s5.1-1} also holds for $\gamma$.

If equality holds in \eqref{s5.1-1} for a convex curve $\gamma$. We need to show that $\gamma$ is a geodesic circle. Let
\begin{equation*}
\gamma_+=\{p\in \gamma, ~\kappa(p)>0\}.
\end{equation*}
Since $\gamma$ is a closed curve, it must have at least one convex point. Therefore, $\gamma_+$ is a non-empty open subset of $\gamma$. We claim that $\gamma_+$ is also closed. In fact, let $\varphi\in C_c^\infty(\gamma_+)$ be a smooth function with compact support in $\gamma_+$.
Following \cite{GL}, we consider a variation of $\gamma$ by
\begin{equation*}
\frac{\partial }{\partial t}X=-\varphi \nu.
\end{equation*}
For sufficiently small time $t\in (-\varepsilon, \varepsilon)$, $\gamma_t=X(\mathbb{S}^1, t)$ is convex, as the variation only deforms a compact subset of $\gamma_+$ for sufficiently small time and $\gamma_+$ is strictly convex. Then $Q(\gamma_t)\ge 0$ for $t\in (-\varepsilon, \varepsilon)$ with $Q(0)=0$. This implies that
\begin{align*}
0= & \frac{d}{dt}\bigg|_{t=0}Q(t) \\
=& \frac{d}{dt}\bigg|_{t=0}\left(L^2-4\pi A+KA^2\right)\\
=& -2L\int_{\gamma}\varphi\kappa ds+4\pi \int_{\gamma} \varphi ds-2KA\int_{\gamma} \varphi ds\\
=& -2\int_{\gamma}\left(L\kappa-2\pi+KA\right)\varphi ds
\end{align*}
for any $\varphi\in C_c^\infty(\gamma_+)$. In particular,
\begin{equation*}
\kappa =\frac{2\pi-KA}{L}
\end{equation*}
on $\gamma_+$, which is a closed condition and hence $\gamma_+$ is a closed subset of $\gamma$. By connectedness of $\gamma$, we conclude that $\gamma_+=\gamma$ and $\gamma$ is strictly convex everywhere. Therefore, by the equality case of \textbf{\underline{Step 1}}, we obtain that $\gamma$ is a geodesic circle.
\end{proof}

\begin{thm}
\label{TMiso}
Let $\gamma$ be a smooth, simple closed curve in $M^2(K)$ with $K=-1$, $1$ and $0$. Then the isoperimetric inequality 
\begin{equation}\label{s5.1-2}
L^2\ge 4\pi A-KA^2
\end{equation}
holds for $\gamma$. Equality holds if and only if $\gamma$ is a geodesic circle.
\end{thm}
\begin{proof}
We will employ the \emph{strictly minimizing hull} $\Omega^*$ of the domain $\Omega$ enclosed by $\gamma$. We first recall the definition and regularity of $\Omega^*$ (see pages 371-372 in \cite{HI01}) of a bounded domain in a complete Riemannian manifold $M$: We call $E$ is a \emph{minimizing hull} if $E$ minimizes area on the outside, that is, if for any $F$ containing $E$ we have
\begin{equation}\label{s5.1-3}
|\partial E|\le |\partial F|.
\end{equation}
$E$ is a \emph{strictly minimizing hull} if equality holds in \eqref{s5.1-3} if
	$|F\setminus E|=0$. For a bounded domain $\Omega\Subset M$, we
	define its strictly minimizing hull $\Omega^*$ as the intersection of
	all strictly minimizing hulls that contain $\Omega$. There hold
	$|\partial \Omega^*|\le |\partial \Omega|$ and $|\Omega|\le
	|\Omega^*|$. If $\partial\Omega$ is $C^2$, the boundary of
	$\partial\Omega^*$ has $C^{1, 1}$ regularity. For the weak mean
	curvature, we have (see \cite[p.372]{HI01})
\begin{align}
H_{\partial\Omega^*} =& 0, \quad \mathrm{on}~\partial\Omega^*\setminus\partial\Omega \nonumber\\
H_{\partial\Omega^*}=&H_{\partial\Omega}\ge 0\quad \mathrm{a.e. ~on}~ \partial\Omega^*\cap \partial\Omega. \label{s5.1-4}
\end{align}

Now let $\Omega$ be the bounded domain enclosed by a smooth simple closed curve $\gamma$ in $M^2(K)$ with $K=-1$, $1$ and $0$. We consider its strictly minimizing hull $\Omega^*$ which has a $C^{1, 1}$ boundary $\gamma^*=\partial\Omega^*$. Since geodesic balls are strictly minimizing hulls in $M^2(K)$, we see that $\Omega^*$ is contained in $M^2(K)$ (Note that for $K=1$ we only consider the curve lies in the hemisphere $M^2(1)=\S_+^2$). Then
\begin{equation}\label{s5.1-5}
L[\gamma]\ge L[\gamma^*], \qquad A[\Omega]\le A[\Omega^*],
\end{equation}
which implies that
\begin{equation*}
Q(\gamma)\ge Q(\gamma^*).
\end{equation*}
To show \eqref{s5.1-2}, it suffices to show $Q(\gamma^*)\ge 0$.

By \eqref{s5.1-4}, the weak curvature of the boundary $\gamma^*=\partial\Omega^*$ is nonnegative a.e. and satisfies
\begin{align*}
\kappa_{\gamma^*}= &0, \quad \mathrm{on}~\gamma^*\setminus \gamma\\
\kappa_{\gamma^*}=&\kappa_\gamma\ge 0\quad \mathrm{a.e. ~on}~ \gamma^*\cap\gamma.
\end{align*}
That is, $\gamma^*$ is a convex curve with $C^{1, 1}$ regularity. Intuitively, the strictly minimizing hull $\Omega^*$ could be obtained by the intersection of all convex domains containing $\Omega$ and it should have convex boundary. To show \eqref{s5.1-2} for a $C^{1, 1}$ convex curve $\gamma^*$, we apply an approximation argument. By employing the curve shortening flow \eqref{s5.1-csf} as in Huisken-Ilmanen's work \cite[Lemma 2.6]{HI08}, $\gamma^*$ can be approximated by a family of smooth, strictly convex curve $\gamma_{\varepsilon}$ with $\varepsilon\in (0, \varepsilon_0]$ from interior in $C^{1, \alpha}\cap W^{2, p}$ for all $0<\alpha<1$, $1\le p<\infty$. Here, $\gamma_\varepsilon$ solves the curve shortening flow \eqref{s5.1-csf}. The proof of Lemma 2.6 in \cite{HI08} was presented in Euclidean space, but the proof can be easily adapted to curved ambient space as the ambient curvature only produces some lower order terms (see also \cite[Lemma 6.8]{AFM20} for a precise statement). Then for each $\gamma_{\varepsilon}$, there holds
\begin{equation*}
Q(\gamma_\varepsilon)\ge 0.
\end{equation*}
Letting $\varepsilon\to 0$, we have
\begin{equation*}
Q(\gamma^*)\ge 0.
\end{equation*}

Finally, if equality holds in \eqref{s5.1-2} for a smooth simple closed
$\gamma$, then equality holds in \eqref{s5.1-5}. This implies that
$\Omega=\Omega^*$ and is outer-minimizing. Then $\Omega$ is smooth and convex.
By the rigidity part of Theorem \ref{s5.thm1}, we conclude that $\gamma$ is a
geodesic circle.
\end{proof}

\subsection{Weighted geometric inequalities}

\begin{lem}\label{lem Phi'}
Let $\Phi(r)$ be defined by \eqref{Phi}.
Along the flow \eqref{flow} starting from an initially strictly convex curve, we have
\begin{align*}
\frac{d}{d t} \left(\int_{\gamma_t} \Phi(r) \kappa d \mu +A[\gamma_t]\right)\le 0
\end{align*}
with equality if and only if $\gamma_t$ is a geodesic circle centered at the origin.
\end{lem}
\begin{proof}
We have
\begin{equation}\begin{split}\label{Phi ss}
\frac{\partial^{2}}{\partial s^{2}} \Phi(r) =\frac{\partial}{\partial s}\langle \overline {\nabla} \Phi(r), T\rangle
=\frac{\partial}{\partial s}\langle V, T\rangle
&=\left\langle \overline {\nabla}_{T} V, T\right\rangle+\left\langle V, \nabla_{s} T\right\rangle \\
&=\phi^{\prime}(r)\langle T, T\rangle+\langle V, -\kappa \nu\rangle \\
&=\phi^{\prime}(r)-\kappa u.
\end{split}\end{equation}
We compute
$\frac{\partial}{\partial t} \Phi=\left\langle\overline {\nabla} \Phi, \partial_{t} X\right\rangle=F u $,
$\frac{\partial}{\partial t} d s=F \kappa d s$ and
\begin{align}\label{us}
\frac{\partial u}{\partial s}= \frac{\partial }{\partial s}\langle V, \nu\rangle = \kappa\langle V, T\rangle =\kappa \frac{\partial \Phi}{\partial s}.
\end{align}

So by \eqref{k'},
\begin{align*}
\frac{d}{d t} \int_{\gamma_t } \Phi \kappa d s=&\int_{\gamma_t } F \kappa u d s+\int_{\gamma_t } \Phi F \kappa^{2} d s
+\int_{\gamma_t } \Phi\left(-\frac{\partial^{2}F}{\partial s^{2}} -(\kappa^{2}+K) F\right) d s\\
=& \int_{\gamma_t }\left(-\frac{\partial^{2}}{\partial s^{2}} \Phi+\kappa u -K \Phi\right) F d s \\
=& \int_{\gamma_t }\left(2 \kappa u -\phi^{\prime}(r)-K \Phi\right) F d s \\
=& \int_{\gamma_t }(2 \kappa u -1) F d s.
\end{align*}
Therefore, by \eqref{A'}, \eqref{Phi ss}, \eqref{us} and Proposition \ref{prop k},
\begin{align*}
\frac{d}{d t} \left(\int_{\gamma_t } \Phi \kappa d s+A\right)
=&\int_{\gamma_t } 2 \kappa u Fd s \\
=& 2 \int_{\gamma_t } u\left(\phi^{\prime}(r)- \kappa u\right) d s\\
=&2 \int_{\gamma_t } u \frac{\partial^{2}\Phi}{\partial s^{2}} d s\\
=&-2 \int_{\gamma_t } \frac{\partial u}{\partial s} \frac{\partial \Phi}{\partial s} d s \\
=&-2 \int_{\gamma_t } \kappa\left|\partial_{s} \Phi\right|^{2} d s \\
\le & 0.
\end{align*}
The equality holds if and only if $\frac{\partial \Phi}{\partial s}=0$ and so $r$ is constant.
\end{proof}

Now let us prove Theorem \ref{thmintPhi}.

\begin{proof}[Proof of Theorem \ref{thmintPhi}]
Firstly, if $\gamma$ is smooth, closed and strictly convex, along the flow \eqref{flow} starting from $\gamma$, $\gamma_t$ converges to a geodesic circle $\gamma_\infty$ of radius $\rho_\infty$ and
\begin{align*}
L[\gamma_t]=L\left[\gamma_\infty\right]=2 \pi \phi(\rho_{\infty})\\
\lim_{t \rightarrow \infty} A\left[\gamma_t\right]=A\left[\gamma_\infty\right]=2 \pi \Phi(\rho_\infty).
\end{align*}
On the other hand, Lemma \ref{lem Phi'} implies
\begin{align*}
\int_{\gamma_t} \Phi \kappa d s +A[\gamma_t]& \ge \displaystyle \lim_{t \rightarrow \infty} \int_{\gamma_t} \Phi \kappa d s+A [\gamma_\infty] =\int_{\gamma_\infty} \Phi \kappa d s +A[\gamma_\infty]\\
&=2 \pi \phi'(\rho_\infty) \Phi(\rho_\infty)+2 \pi \Phi\left(\rho_{\infty}\right).
\end{align*}
By direct checking, $\Phi \Phi^{\prime \prime}+\Phi={\Phi^{\prime }}^2$, i.e. $\Phi \phi'+\Phi=\phi^2$.
Therefore
\begin{align*}
\int_{\gamma} \Phi \kappa d s+A[\gamma]
\ge 2 \pi \phi(\rho_\infty)^{2}
=\frac{1}{2 \pi} L^{2}.
\end{align*}
If equality holds in \eqref{s5.2-1} for a strictly convex curve $\gamma$, by Lemma \ref{lem Phi'} we have that $\gamma$ is a geodesic circle centered at the origin.

In general, the inequality \eqref{s5.2-1} and its rigidity also hold for smooth closed convex curve $\gamma$, by applying a similar approximation argument as in \underline{\textbf{Step 2}} of the proof of Theorem \ref{s5.thm1}.
\end{proof}

By Gauss-Bonnet theorem, we obtain the following inequalities.
\begin{cor}
\begin{enumerate}
\item
Let $\gamma \subset \mathbb H^{2}$ be a smooth, closed and convex curve. Then
\begin{align}\label{int k cosh}
\int_{\gamma} \kappa\cosh r \ge 2 \pi+\frac{L^{2}}{2 \pi}.
\end{align}
Equality holds if and only if $\gamma$ is a geodesic circle centered at the origin.
\item
Let $\gamma \subset \mathbb S^{2}_+$ be a smooth, closed and convex curve. Then
\begin{align}\label{int k cos}
\int_{\gamma} \kappa\cos r \le 2 \pi-\frac{L^{2}}{2 \pi}.
\end{align}
\end{enumerate}
Equality holds if and only if $\gamma$ is a geodesic circle centered at the origin.
\label{CY1}
\end{cor}
\begin{rmk}
The inequality \eqref{int k cosh} extends \cite[Theorem 1]{dLG} to $n=2$.
\end{rmk}
By the isoperimetric inequality, Theorem \ref{thmintPhi} implies that for a convex curve $\gamma\subset \mathbb H^2$,
\begin{align}\label{s5.2-2b}
\int_{\gamma} \Phi(r) |\kappa| \ge A+\frac{A^{2}}{2 \pi}.
\end{align}

We show that this inequality also holds for smooth simple closed curves.

\begin{thm}
Let $\gamma\subset \mathbb{H}^2$ be a smooth, simple closed curve. Then
\begin{equation} \label{s5.2-3c}
\int_\gamma \Phi(r)|\kappa|\,ds \ge A+\frac{A^2}{2\pi}
\end{equation}
holds for $\gamma$, where $\Phi(r)=\cosh r-1$. Equality holds if and only if $\gamma$ is a geodesic circle centered at the origin.
\end{thm}
\begin{proof}
Let $\Omega$ be the bounded domain in $\mathbb{H}^2$ enclosed by a smooth simple closed curve $\gamma$. Then the strictly minimizing hull $\Omega^*$ of $\Omega$ has a $C^{1, 1}$ boundary $\gamma^*=\partial\Omega^*$. We have
\begin{equation*}
L[\gamma]\ge L[\gamma^*], \qquad A[\Omega]\le A[\Omega^*].
\end{equation*}
The weak curvature of $\gamma^*$ satisfies
\begin{align*}
\kappa_{\gamma^*}= &0, \quad \mathrm{on}~\gamma^*\setminus \gamma\\
\kappa_{\gamma^*}=&\kappa_\gamma\ge 0\quad \mathrm{a.e. ~on}~ \gamma^*\cap\gamma
\end{align*}
which implies that
\begin{align*}
\int_\gamma \Phi |\kappa|d\mu=& \int_{\gamma\cap\gamma^*}\Phi \kappa d\mu+\int_{\gamma\setminus\gamma^*}\Phi |\kappa|d\mu\\
\ge & \int_{\gamma\cap\gamma^*}\Phi \kappa d\mu+\int_{\gamma^*\setminus\gamma}\Phi \kappa d\mu\\
=& \int_{\gamma^*}\Phi \kappa d\mu.
\end{align*}
Since $\gamma^*$ is convex, we have that
\begin{align*}
\int_{\gamma^*}\Phi \kappa d\mu
\ge & A[\Omega^*]+\frac{A^2[\Omega^*]}{2\pi}\\
\ge & A[\Omega]+\frac{A^2[\Omega]}{2\pi}.
\end{align*}
If equality holds in \eqref{s5.2-3c}, then $\Omega=\Omega^*$ is outerminimizing, with smooth convex boundary. The rigidity of \eqref{s5.2-2b} for smooth convex case implies that $\gamma=\partial \Omega$ is a geodesic circle centered at the origin.
\end{proof}

The inequality \eqref{int k cos} is related to the following conjecture in \cite{girao2017alexandrov}:
\begin{conjecture}
\label{conj1}
Let $\Sigma \subset \mathbb S^{n}$ be a closed, orientable and connected embedded hypersurface. If $\Sigma$ is strictly convex, then
\begin{align*}
\sup_{x \in \mathbb S^{n}} \int_{\Sigma} \rho_{x} H d \Sigma \ge (n-1) \omega_{n-1}\left[\left(\frac{|\Sigma|}{\omega_{n-1}}\right)^{\frac{n-2}{n-1}}-\left(\frac{|\Sigma|}{\omega_{n-1}}\right)^{\frac{n}{n-1}}\right].
\end{align*}
Here $\rho_{x}(q)=\cos \left(d_{\mathbb S^n}(x, q)\right)$ and $d_{\mathbb S^n}(x, q)$ denotes the geodesic distance from $q$ to $x$.
Moreover, the equality holds if and only if $\Sigma$ is a geodesic sphere.
\end{conjecture}
We show here that this conjecture is false when $n=2$.
\begin{cor}
\label{CY2}
There is a curve $\gamma:\S^1\rightarrow\S^2$ such that
\[
	\int_{\gamma} \kappa \cos r d s < 2 \pi-\frac{L^{2}}{2 \pi}\,,
\]
that is, Conjecture \ref{conj1} is not true.
\end{cor}
\begin{proof}
When $n=2$, the conjecture becomes
\begin{align*}
\max_{y \in \mathbb S^2} \int_\gamma \cos \left(d_{\mathbb S^2}(y, x)\right) \kappa(x) ds(x) \ge 2 \pi-\frac{L^{2}}{2 \pi}.
\end{align*}

Let $\gamma$ be a general strictly convex curve in $\mathbb S^2$.
Consider the function $F: \mathbb S^2 \to \mathbb R$ defined by
\begin{align*}
F(y)=\int_\gamma \kappa(x) \cos \left(d_{\mathbb S^2}(x, y) \right)ds(x) =\int_\gamma \kappa(x) \langle x, y\rangle_{\mathbb R^3} ds(x).
\end{align*}
Suppose the maximum of this function is attained at $y_0$, then by Lagrange multiplier method, we have $$\nabla^{\mathbb R^3} F(y_0)=\lambda y_0. $$
Here $\nabla^{\mathbb R^3}$ is the gradient in $\mathbb R^3$. As $\nabla^{\mathbb R^3}F=\int_{\gamma} \kappa(x) x ds(x)$, the maximum point is attained at the point $y_0\in \mathbb S^2$ which is parallel to $\int_\gamma \kappa x ds$ (it is easy to see that $\int_{\gamma} \kappa x ds(x)\ne0$ for a strictly convex curve $\gamma$). It is also easy to see that $\lambda>0$.

Fix a geodesic circle $C$ which lies in the open hemisphere centered at a point $o\in \mathbb S^2$. We can perturb $C$ to a curve $\gamma$ such that
\begin{enumerate}
\item
$\gamma$ is not a geodesic circle.
\item
$\gamma$ is still strictly convex, $v:=\int_\gamma \kappa (x)x ds=|v| y_0$, $\gamma$ lies in the open hemisphere centered at $y_0$, and $y_0\in \mathbb S^2$ is still enclosed by $\gamma$.
\end{enumerate}

This can be done by expressing $\gamma$ as the graph of a small $C^2$ perturbation of the constant function $r_0$, where $ r_0$ is the radius of $C$.

By the analysis above, for this curve $\gamma$, the maximum of $F$ is attained
	at $y_0$. Let $r=d_{\mathbb S^2}(y_0, \cdot)$. It then follows from
	Corollary 5.5 that $\int_{\gamma} \kappa \cos r d s < 2
	\pi-\frac{L^{2}}{2 \pi}$. i.e. the conjecture is not true.
\end{proof}

\printbibliography

\end{document}